\documentclass[12pt]{amsart}

\allowdisplaybreaks

\def\[#1\]{\begin{align}#1\end{align}}
\def\*[#1\]{\begin{align*}#1\end{align*}}

\newcommand{\st}{\,:\,}

\DeclareMathOperator*{\newlim}{\mathrm{lim}\vphantom{\mathrm{infsup}}}
\DeclareMathOperator*{\newmin}{\mathrm{min}\vphantom{\mathrm{infsup}}}
\DeclareMathOperator*{\newmax}{\mathrm{max}\vphantom{\mathrm{infsup}}}
\DeclareMathOperator*{\newinf}{\mathrm{inf}\vphantom{\mathrm{infsup}}}

\renewcommand{\lim}{\newlim}
\renewcommand{\min}{\newmin}
\renewcommand{\max}{\newmax}
\renewcommand{\inf}{\newinf}

\def \ProbUnion{P\left(\bigcup_{i=1}^N A_i\right)}

\def \uNEWI{\hbar_{\textrm{NEW-1}}}

\def \lKAT{\ell_{\textrm{KAT}}}

\def \lDS{\ell_{\textrm{DS}}}
\def \lDC{\ell_{\textrm{DC}}}
\def \lGK{\ell_{\textrm{GK}}}

\def \uGREEDY{\hbar_{\textrm{Greedy}}}

\def \bk{\boldsymbol{k}}

\def \bc{\boldsymbol{c}}

\def \balpha{\boldsymbol{\alpha}}

\def \bzero{\boldsymbol{0}}

\def \bSigma{\boldsymbol{\Sigma}}

\def \ba{\boldsymbol{a}}
\def \bb{\boldsymbol{b}}

\def \bL{\boldsymbol{L}}

\def \bQ{\boldsymbol{Q}}

\def \st{\textrm{s.t. }}

\newcommand{\bbeta}{\boldsymbol{\beta}}

\newcommand{\bx}{\boldsymbol{x}}

\newcommand{\bA}{\boldsymbol{A}}

\usepackage[foot]{amsaddr}
\usepackage[utf8]{inputenc} 
\usepackage{amsmath, amssymb,bm, cases, mathtools, thmtools}
\usepackage{verbatim}
\usepackage{graphicx}\graphicspath{{figures/}}
\usepackage{multicol}
\usepackage{tabularx}
\usepackage[usenames,dvipsnames]{xcolor}
\usepackage{mathrsfs} 
\usepackage{url}

\usepackage[%
    minnames=1,maxnames=99,maxcitenames=3,
    style=alphabetic,
    doi=false,url=false,
    firstinits=true,hyperref,natbib,backend=bibtex]{biblatex}
\renewbibmacro{in:}{%
  \ifentrytype{article}{}{\printtext{\bibstring{in}\intitlepunct}}}

\bibliography{Union_Bounds}
\usepackage[colorlinks,citecolor=blue,urlcolor=blue,linkcolor=RawSienna]{hyperref}
\usepackage{hypernat}
\usepackage{datetime}

\usepackage{euscript,microtype}

\usepackage[capitalize]{cleveref}

\crefname{lemma}{Lemma}{Lemmas}
\crefname{corollary}{Corollary}{Corollaries}
\crefname{theorem}{Theorem}{Theorems}

\makeatletter
\let\reftagform@=\tagform@
\def\tagform@#1{\maketag@@@{\ignorespaces\textcolor{gray}{(#1)}\unskip\@@italiccorr}}
\renewcommand{\eqref}[1]{\textup{\reftagform@{\ref{#1}}}}
\makeatother

\declaretheorem[style=plain,numberwithin=section,name=Theorem]{theorem}
\declaretheorem[style=plain,sibling=theorem,name=Lemma]{lemma}

\declaretheorem[style=definition,sibling=theorem,name=Definition]{definition}

\declaretheorem[style=remark,qed=$\triangleleft$,sibling=theorem,name=Remark]{remark}

\numberwithin{theorem}{section}

\begin{document}


\title[A Short Survey]{A Short Survey on Bounding the Union Probability Using Partial Information}

\author{Jun Yang}
\address{Department of Statistical Sciences\\ University of Toronto, Canada}
\email{jun@utstat.toronto.edu}

\author{Fady Alajaji}
\address{Department of Mathematics and Statistics\\ Queen's University, Canada}
\email{fa@queensu.ca}

\author{Glen Takahara}
\email{takahara@mast.queensu.ca}

\begin{abstract}
	This is a short survey on existing upper and lower bounds on the probability of the union of a finite number of events using partial information given in terms of the individual or pairwise event probabilities (or their sums). New proofs for some of the existing bounds are provided and new observations regarding the existing Gallot--Kounias bound are given.
\end{abstract}

\maketitle

	\begin{center}
	\begin{minipage}{1\linewidth}
		\setcounter{tocdepth}{1}
		\tableofcontents
	\end{minipage}
	\end{center}
\newpage
\allowdisplaybreaks

\section{Introduction}
%
%
Consider a finite family of events $\{A_1,\dots,A_N\}$ in a general probability space $(\Omega ,\mathscr{F},P)$, where $N$ is a fixed positive integer. Note that there are only finitely many Boolean atoms\footnote{The problem can be directly reduced to the finite probability space case. Thus,
	 we will consider finite probability spaces where $\omega\in\Omega$ denotes an elementary outcome instead of an atom.} specified by the $A_i$'s \cite{DeCaen1997}. We are interested in bounding the probability of the finite union of events, i.e., $\ProbUnion$, in terms of partial probabilistic event information such as knowing the individual event probabilities, $\{P(A_1),\dots,P(A_N)\}$, and the pairwise event probabilities $\{P(A_i\cap A_j), i\neq j\}$, or (linear) functions of the probabilities of individual and pairwise events.

For example, the well-known union upper bound and the Bonferroni inequality \cite{GalambosBook} are respectively given as follows:
\begin{equation}\label{union_upper_bound}
  \ProbUnion\le \sum_{i=1}^N P(A_i),
\end{equation}
\begin{equation}\label{bonferroni_lower_bound}
  \ProbUnion\ge \sum_{i=1}^N P(A_i)-\sum_{i<j} P(A_i\cap A_j).
\end{equation}
We note that the union upper bound (\ref{union_upper_bound}) is established in terms of only $\sum_{i=1}^N P(A_i)$ so that each of the individual event probability $P(A_i)$ is actually not needed. However, the Bonferroni lower bound (\ref{bonferroni_lower_bound}) is established using two terms, $\sum_{i=1}^N P(A_i)$ and $\sum_{i<j} P(A_i\cap A_j)$. Therefore, the union upper bound (\ref{union_upper_bound}) and the Bonferroni inequality (\ref{bonferroni_lower_bound}) are established based on different partial information on the event probabilities.

In order to distinguish the use of different partial information, we assume that a vector $\theta=(\theta_1,\dots,\theta_m)\in\mathbb{R}^m$ represents partial probabilistic information about the union $\bigcup_{i=1}^N A_i$. Specifically, we assume that for a given integer $m\ge 1$, $\Theta$ denotes the range of a function of $P(A_i)$'s and $P(A_i\cap A_j)$'s, $\eta_m:[0,1]^{N+\binom{N}{2}}\rightarrow\mathbb{R}^m$. Then $\theta$ equals to the value of the function $\eta_m$ for given $A_1,\dots,A_N$. For example,
\begin{equation}\label{}
  \theta=\left(P(A_1),P(A_2),\dots,P(A_N)\right),
\end{equation}
or
\begin{equation}\label{}
  \theta=\left(\sum_{i=1}^N P(A_i),\sum_{i<j} P(A_i\cap A_j)\right).
\end{equation}
Then, we can define a lower bound (and similarly an upper bound) on $\ProbUnion$ that is established using the partial information represented by $\theta$ as follows.
\begin{definition}
A lower bound of $\ProbUnion$ is a function of $\theta$, $\ell(\theta)$, such that
\begin{equation}\label{}
  \ProbUnion\ge\ell(\theta),
\end{equation}
for any set of events $\{A_i\}$ that the value of $\eta_m$ for given $\{A_i\}$ equals to $\theta$.
\end{definition}

Note that, for given $\theta$, such as $\theta=\left(P(A_1),\dots, P(A_N)\right)$, there are multiple functions of $\theta$ that are lower bounds, for example,
\begin{equation}\label{}
  \begin{split}
  \ProbUnion&\ge\theta_1=P(A_1),\\
  \ProbUnion&\ge\frac{\sum_i\theta_i}{N}=\frac{\sum_i P(A_i)}{N},\\
  \ProbUnion&\ge\max_i\theta_i=\max_i P(A_i).
  \end{split}
\end{equation}
Therefore, we need to define an \emph{optimal} lower bound in a general class of lower bounds that are functions of $\theta$.

Let $\mathscr{L}_{\Theta}$ denote the set of all lower bounds on $P\left(\bigcup_{i=1}^N A_i\right)$ that are functions of only $\theta$.

\begin{definition}
We say that a lower bound $\ell ^{\star}\in\mathscr{L}_{\Theta}$ is
\textit{optimal in $\mathscr{L}_{\Theta}$} if $\ell ^{\star}(\theta)\ge\ell (\theta )$
for all $\theta\in\Theta$ and $\ell\in\mathscr{L}_{\Theta}$.
\end{definition}

\begin{definition}
We say that a lower bound $\ell\in\mathscr{L}_{\Theta}$ is
\textit{achievable} if for every $\theta\in\Theta$,
\begin{equation}\label{def_achievable}
 \inf_{A_1,\ldots ,A_N}P\left(\bigcup_{i=1}^N A_i\right)=\ell (\theta ),
\end{equation}
where the infimum ranges over all collections $\{A_1,\ldots ,A_N\}$,
$A_i\in\mathscr{F}$, such that $\{A_1,\ldots ,A_N\}$ is represented by
$\theta$.
\end{definition}

For bounds in $\mathscr{L}_{\Theta}$, the following lemma shows that achievability is equivalent to optimality.
\begin{lemma}\label{optimality_lemma}
A lower bound $\ell ^{\star}\in\mathscr{L}_{\Theta}$ is optimal
in $\mathscr{L}_{\Theta}$ if and only if it is achievable.
\end{lemma}
\begin{proof}
Suppose that $\ell ^{\star}$ is achievable. Let $\theta\in\Theta$
and $\epsilon >0$ be given, and let $\ell$ be any lower bound in
$\mathscr{L}_{\Theta}$. By achievability there exist sets $A_1,\ldots ,A_N$
in $\mathscr{F}$ represented by $\theta$ such that
\[
\ell ^{\star}(\theta )>P\left(\bigcup_{i=1}^N A_i\right)-\epsilon \ge\ell (\theta )-\epsilon .
\]
Since this holds for any $\epsilon$ we have
$\ell ^{\star}(\theta )\ge\ell (\theta )$. We prove the converse by the
contrapositive. Suppose that $\ell ^{\star}$ is not achievable. Then there
exists $\theta '\in\Theta$ such that
\[
\inf_{A_1,\ldots ,A_N}P\left(\bigcup_{i=1}^N A_i\right)>\ell ^{\star}(\theta '),
\]
where the infimum ranges over all collections $\{A_1,\ldots ,A_N\}$,
$A_i\in\mathscr{F}$, such that $\{A_1,\ldots ,A_N\}$ is represented by
$\theta '$. Define $\ell$ by
\[
\ell (\theta ) =\left\{\begin{array}{cl}
c & \mbox{if $\theta =\theta '$} \\
0 & \mbox{if $\theta\ne\theta '$,}
\end{array}\right.
\]
where $c$ satisfies
\[
\inf_{A_1,\ldots ,A_N}P\left(\bigcup_{i=1}^N A_i\right)>c>\ell ^{\star}(\theta ').
\]
Then $\ell\in\mathscr{L}_{\Theta}$ and is larger than $\ell ^{\star}$ at
$\theta '$. Hence, $\ell ^{\star}$ is not optimal.
\end{proof}

Using Lemma \ref{optimality_lemma}, we can therefore prove that a lower bound $\ell(\theta)$ is optimal if for any value of $\theta\in\Theta$, one can construct a collection of events $\{A_i^*\}$ that is represented by $\theta$ and $P\left(\bigcup_{i=1}^N A_i^*\right)=\ell(\theta)$. The optimal upper bound can also be defined similarly (using a supremum in (\ref{def_achievable})) and proved by achievability.
For example, one can easily verify the following by a construction proof of achievability.
\begin{itemize}
\item $\ProbUnion\ge\frac{\sum_i P(A_i)}{N}$ is the optimal lower bound in the class for $\theta=\left(\sum_i P(A_i)\right)$.
\item $\ProbUnion\ge\max_i P(A_i)$ is the optimal lower bound in the class for $\theta=\left(P(A_1),\dots,P(A_N)\right)$.
\item  $\ProbUnion\le\min\{\sum_i P(A_i),1\}$ is the optimal upper bound in the classes for both $\theta=\left(\sum_i P(A_i)\right)$ and $\theta=\left(P(A_1),\dots,P(A_N)\right)$.
\end{itemize}

Furthermore, we can prove that a lower bound is not optimal by showing it is not achievable. For example, in order to show that the Bonferroni inequality (\ref{bonferroni_lower_bound}) is not an optimal lower bound in the class of lower bounds that are functions of $\theta=\left(\sum_i P(A_i),\sum_{i<j} P(A_i\cap A_j)\right)$, we only need to show it is not achievable. Note that for $N>3$, the lower bound (\ref{bonferroni_lower_bound}) can have negative values. However, according to the definition of achievability, the LHS of (\ref{def_achievable}) can never be negative, which means the lower bound (\ref{bonferroni_lower_bound}) cannot be achievable. Therefore, the Bonferroni inequality (\ref{bonferroni_lower_bound}) is not optimal.


Throughout the survey, we mainly focus on lower bounds using different partial probabilistic information. Upper bounds are presented as remarks.

\allowdisplaybreaks

\section{Review of Existing Bounds}\label{ch:Background}
We start from the class of lower bounds in terms of $\sum_i P(A_i)$ and $\sum_{i<j} P(A_i\cap A_j)$, for which the Dawson-Sankoff (DS) lower bound \cite{Dawson1967} is known as optimal. Then we introduce some lower bounds in terms of $\{P(A_i)\}$ and $\{\sum_{j} P(A_i\cap A_j)\}$, including the D. de Caen (DC) bound \cite{DeCaen1997} and the Kuai-Alajaji-Takahara (KAT) bound \cite{Kuai2000}. Next, a review of some lower bounds in terms of $\{P(A_i)\}$ and $\{P(A_i\cap A_j)\}$ is given, including the algorithmic stepwise lower bound \cite{Kuai2000a} and the Gallot-Kounias (GK) bound \cite{Gallot1966,Kounias1968,Feng2010}. Finally, some existing upper bounds are reviewed, including the Hunter upper bound and the algorithmic greedy upper bound \cite{Kuai2000a}.

We first define the degree of an atom (or outcome in finite probability space) $\omega\in\mathscr{F}$ as follows.
\begin{definition}
For each atom $\omega\in\mathscr{F}$, let the degree of $\omega$, denoted by $\deg(\omega)$, be the number of $A_i$'s that contain $\omega$.
\end{definition}
Therefore, the degree of any atom in $\bigcup_i A_i$ equals to an integer in $\{1,\dots,N\}$.
\subsection{Lower Bounds Using ${\sum_i P(A_i)}$ and ${\sum_{i<j}P(A_i\cap A_j)}$}
Considering $\theta=\left(\sum_i P(A_i),\sum_{i<j}P(A_i\cap A_j)\right)$, we note that the Bonferroni inequality (\ref{bonferroni_lower_bound}) is a lower bound in this class. However, we have shown that (\ref{bonferroni_lower_bound}) is not optimal, which means there exists another function of only $\sum_i P(A_i)$ and $\sum_{i<j}P(A_i\cap A_j)$ that is a lower bound of $\ProbUnion$ and always sharper than the Bonferroni inequality.

Defining
\begin{equation}\label{def_a_k}
 a(k):=P\left(\left\{\omega\subseteq\bigcup_i A_i, \deg(\omega)=k\right\}\right),
\end{equation}
one can easily verify the following identities:
\begin{equation}\label{}
  \ProbUnion=\sum_{k=1}^N a(k),
\end{equation}
\begin{equation}\label{}
  \sum_{i=1}^N P\left(A_i\right)=\sum_{k=1}^N k a(k),
\end{equation}
\begin{equation}\label{}
  \sum_{i<j} P(A_i\cap A_j)=\sum_{k=2}^N \binom{k}{2}a(k),
\end{equation}
\begin{equation}\label{}
\begin{split}
  \sum_{i,j}P\left(A_i\cap A_j\right)&=2\sum_{i<j} P(A_i\cap A_j)+\sum_{i=1}^N P(A_i)\\
  &=2\sum_{k=2}^N\frac{k(k-1)}{2}a(k)+\sum_{k=1}^N k a(k)\\
  &=\sum_{k=1}^N k^2 a(k).
\end{split}
\end{equation}

Note that using the above equalities, one can derive a lower bound simply via the Cauchy-Schwarz inequality:
\begin{equation}\label{}
  \left(\sum_k a(k)\right)\left(\sum_k k^2 a(k)\right)\ge\left(\sum_k k a(k)\right)^2,
\end{equation}
where equality holds if and only if $a(k)>0$ only for a particular $k$, i.e., all outcomes in the union has the same degree $k$. The resulting lower bound can be written as
\begin{equation}\label{new1_lower_bound}
  \ProbUnion\ge \frac{\left(\sum_i P(A_i)\right)^2}{\sum_{i,j}P(A_i\cap A_j)}.
\end{equation}
Since $\sum_{i,j}P(A_i\cap A_j)\le\sum_{i,j} P(A_i)=N\sum_i P(A_i)$,
\begin{equation}\label{}
  \frac{\left(\sum_i P(A_i)\right)^2}{\sum_{i,j}P(A_i\cap A_j)}\ge \frac{\sum_i P(A_i)}{N};
\end{equation}
hence the lower bound (\ref{new1_lower_bound}) is always sharper than $\frac{\sum_i P(A_i)}{N}$, which has been shown to be optimal in the class of $\theta=\left(\sum_i P(A_i)\right)$. This is reasonable since the lower bound (\ref{new1_lower_bound}) is established using more information than $\frac{\sum_i P(A_i)}{N}$. However, it can be readily shown that the lower bound (\ref{new1_lower_bound}) is not always sharper than the Bonferroni inequality (\ref{bonferroni_lower_bound}). Therefore, it cannot be the optimal lower bound in the class of $\theta=\left(\sum_i P(A_i),\sum_{i<j}P(A_i\cap A_j)\right)$.

\subsubsection{Dawson-Sankoff (DS) Bound}
The DS bound is known as the optimal lower bound in terms of only $\sum_i P(A_i)$ and $\sum_{i<j}P(A_i\cap A_j)$. Denoting $\theta_1:=\sum_i P(A_i)$ and $\theta_2:=\sum_{i<j}P(A_i\cap A_j)$, the DS bound \cite{Dawson1967} can be written as
\begin{equation}\label{DS_lower_bound}
  \ProbUnion\ge\frac{\kappa\theta_1^2}{(2-\kappa)\theta_1+2\theta_2}+\frac{(1-\kappa)\theta_1^2}{(1-\kappa)\theta_1+2\theta_2},
\end{equation}
where $\kappa=\frac{2\theta_2}{\theta_1}-\lfloor\frac{2\theta_2}{\theta_1}\rfloor$ and $\lfloor x\rfloor$ denotes the largest integer less than or equal to $x$.

We first show that the DS bound is the solution of a linear programming (LP) problem in the following lemma.
\begin{lemma} The DS bound is the solution of the following LP problem.
\begin{equation}\label{LP_DS}
\begin{split}
  \lDS:=&\min_{a(k)} \sum_k a(k),\\
  \st&\quad \sum_k k a(k)=\sum_i P(A_i),\\
  &\quad \sum_k k^2 a(k)=\sum_{i,j} P(A_i\cap A_j),\\
 &\quad a(k)\ge 0, \quad k=1,\dots,N.
\end{split}
\end{equation}
\end{lemma}
\begin{proof}
For an LP problem, when a feasible solution exists and when the objective
function (which is linear) is bounded, the optimal value of the objective function is always attained on the boundary of the optimal level-set
and it is attained on at least one of the vertices of the polyhedron
formed by the constraints (which is the set of feasible solutions) \cite{bertsimas-LPbook}. Then, the lemma can be readily verified using this fact that one of the optimal feasible points of the LP problem (\ref{LP_DS}) is a vertex. To obtain a vertex, one need to make $N-2$ of the inequalities $a(k)\ge 0$ active, which means there are only two integers $k_1$ and $k_2$ that $1\le k_1<k_2\le N$, satisfying
\begin{equation}\label{}
\begin{split}
  \min_{k_1,k_2}\quad &a(k_1)+a(k_2),\\
  \st&\quad k_1 a(k_1)+k_2 a(k_2)=\sum_i P(A_i),\\
  &\quad k_1^2 a(k_1)+ k_2^2 a(k_2)=\sum_{i,j} P(A_i\cap A_j),\\
  &\quad a(k_1)\ge 0,\quad a(k_2)\ge 0.
\end{split}
\end{equation}
It can be easily shown that the solution of the above problem is achieved at $k_1=\lfloor\frac{\sum_{i,j} P(A_i\cap A_j)}{\sum_i P(A_i)}\rfloor$ and $k_2=k_1+1$. Thus, the solution of (\ref{LP_DS}) is the DS bound.
\end{proof}

The existing proof of the optimality of the DS bound can be seen, e.g., in \cite[p. 22]{GalambosBook}. We herein give an alternative and simpler proof by proving it is achievable.
\begin{lemma} The DS bound is optimal in the class of lower bounds in terms of $\theta=\left(\sum_i P(A_i),\sum_{i<j}P(A_i\cap A_j)\right)$.
\end{lemma}
\begin{proof} We have shown that the DS bound is the solution of (\ref{LP_DS}) and can be written as $\lDS=a(k_1)+a(k_2)$, for some $a(k_1)\ge 0$, $a(k_2)\ge 0$, and $a(k)=0, k\neq k_1, k\neq k_2$. Recalling the definition of $a(k)$, one can construct two outcomes $\omega_1$ and $\omega_2$ in a finite probability space such that
\begin{equation}\label{}
  P(\omega_1)=a(k_1),\quad P(\omega_2)=a(k_2).
\end{equation}
Then consider the following construction of collection of events $\{A_i^*\}$,
\begin{equation}\label{}
\begin{split}
  A_i^*&=\{\omega_1,\omega_2\}, \quad\textrm{if}\quad i\le k_1,\quad\\
  A_i^*&=\{\omega_2\},\quad \textrm{if}\quad k_1<i\le k_2,\quad\\
  A_i^*&=\emptyset,\quad \textrm{otherwise}.
\end{split}
\end{equation}
Then we always have $\lDS=P\left(\bigcup_i A_i^*\right)$. Therefore, the DS bound is achievable, and hence optimal.
\end{proof}

Note that since the DS bound is optimal, it is always sharper than the lower bound in (\ref{new1_lower_bound}). Actually, this can be easily proved since the lower bound in (\ref{new1_lower_bound}) is a lower bound of the objective function of (\ref{LP_DS})  by Cauchy-Schwarz inequality using the two constraints of (\ref{LP_DS}).

\subsection{Lower Bounds Using ${\{P(A_i)\}}$ and ${\{\sum_{j\neq i}P(A_i\cap A_j)\}}$}
In this section, we review the lower bounds in terms of $\{P(A_i)\}$ and $\{\sum_{j\neq i}P(A_i\cap A_j)\}$, including the DC \cite{DeCaen1997} and the KAT \cite{Kuai2000} bounds.

Similar to the definition of $a(k)$, define
\begin{equation}\label{def_a_i_k}
  a_i(k):=P\left(\left\{\omega\subseteq A_i: \deg(\omega)=k\right\}\right),\quad i=1,\dots,N,\quad k=1,\dots,N.
\end{equation}
Then one can verify that $\sum_i a_i(k)=k a(k)$, i.e.,
\begin{equation}\label{relation_a_k_a_i_k}
  a(k)=\frac{\sum_i a_i(k)}{k}.
\end{equation}
For simplicity, we denote
\begin{equation}\label{}
  \alpha_i:= P(A_i),\quad\beta_i:=\sum_{j\neq i} P(A_i\cap A_j),\quad \gamma_i:=\alpha_i+\beta_i=\sum_j P(A_i\cap A_j).
\end{equation}
We examine lower bounds that are functions of $\theta=\left(\alpha_1,\dots,\alpha_N,\gamma_1,\dots,\gamma_N\right)$.

One can verify that $\ProbUnion$, $\alpha_i$ and $\gamma_i$ can all be written as linear functions of $\{a_i(k)\}$ as follows.
\begin{equation}\label{}
  \ProbUnion=\sum_k a(k)=\sum_i \sum_k \frac{a_i(k)}{k}.
\end{equation}
\begin{equation}\label{}
  \alpha_i=P(A_i)=\sum_k a_i(k),\quad \gamma_i=\sum_j P(A_i\cap A_j)=\sum_k k a_i(k).
\end{equation}

\subsubsection{D. de Caen (DC) bound}
Similar to the lower bound in (\ref{new1_lower_bound}), using the Cauchy-Schwarz inequality
\begin{equation}\label{}
  \left(\sum_k\frac{a_i(k)}{k}\right)\left(\sum_k ka_i(k)\right)\ge\left(\sum_k a_i(k)\right)^2
\end{equation}
for $i=1,\dots,N$, and summing over $i$, one can get the DC bound as follows.
\begin{equation}\label{DC_bound}
  P\left(\bigcup_i A_i\right)=\sum_i \sum_k \frac{a_i(k)}{k}\ge \sum_i \left(\frac{\alpha_i^2}{\gamma_i}\right)=\sum_i\frac{P(A_i)^2}{\sum_j P(A_i\cap A_j)}=:\lDC.
\end{equation}
It is noted by D. de Caen \cite{DeCaen1997} that the above lower bound can be (but is not always) sharper than the DS bound.

\subsubsection{The Kuai-Alajaji-Takahara (KAT) bound}
Now, we introduce the KAT bound
\begin{equation}\label{def_lKAT}
  \lKAT:=\sum_{i=1}^N\quad \left\{ \left[\frac{1}{\lfloor\frac{\gamma_i}{\alpha_i}\rfloor}-\frac{\frac{\gamma_i}{\alpha_i}-\lfloor\frac{\gamma_i}{\alpha_i}\rfloor}{(1+\lfloor\frac{\gamma_i}{\alpha_i}\rfloor)(\lfloor\frac{\gamma_i}{\alpha_i}\rfloor)}\right]\alpha_i\right\},
\end{equation}
as the solution of an LP problem, which is given in the following Lemma.
\begin{lemma}\label{lemma_KAT}
The KAT bound is the solution of the following LP problem
\begin{equation}\label{LP_KAT}
\begin{split}
    \min_{\{a_i(k),i=1,\ldots,N,k=1,\ldots,N\}}&\quad\sum_{i=1}^N\sum_{k=1}^N\frac{a_i(k)}{k}\\
    \st&\quad\sum_{k=1}^N a_i(k)=\alpha_i,\quad \sum_{k=1}^N k a_i(k)= \gamma_i,\quad i=1,\ldots,N,\\
    &\quad a_i(k)\ge 0,\quad i=1,\ldots,N,\quad k=1,\ldots,N,
\end{split}
\end{equation}
\end{lemma}
\begin{proof}
One can separate each $i$ in the problem (\ref{LP_KAT}) and solve $N$ sub-optimization problems separately for each $i$:
\begin{equation}\label{LP_KAT_subproblem}
\begin{split}
    \min_{\{a_i(k),k=1,\ldots,N\}}&\quad\sum_{k=1}^N\frac{a_i(k)}{k}\\
    \st&\quad\sum_{k=1}^N a_i(k)=\alpha_i,\quad \sum_{k=1}^N k a_i(k)= \gamma_i,\\
    &\quad a_i(k)\ge 0,\quad k=1,\ldots,N.
\end{split}
\end{equation}
Each of the sub-problems can be solved using the same method as solving the LP problem (\ref{LP_DS}) for the DS bound. One can see \cite{Kuai2000} for details. An alternative proof is given in \cite{KuaiThesis} by solving the dual LP problem of (\ref{LP_KAT_subproblem}).
\end{proof}

It has been shown that the KAT bound is always sharper than both the DC bound and the DS bound \cite{Kuai2000}. Furthermore, Dembo has shown \cite{DemboUnpublished} that the KAT bound improves the DC bound by a factor of at most $\frac{9}{8}$. In the following lemma, we give alternative and simpler proofs of the above results.

\begin{lemma}
Comparing with the DC and DS bounds, the KAT bound satisfies
\begin{equation}\label{}
  \max\{\lDC,\lDS\}\le\lKAT\le\frac{9}{8}\lDC.
\end{equation}
\end{lemma}
\begin{proof}
First, substituting (\ref{relation_a_k_a_i_k}) in (\ref{LP_DS}), one can get that the DS bound is the solution of the following LP problem of $\{a_i(k)\}$
\begin{equation}\label{LP_DS_a_i_k}
\begin{split}
    \lDS=&\min_{\{a_i(k),i=1,\ldots,N,k=1,\ldots,N\}}\quad\sum_{i=1}^N\sum_{k=1}^N\frac{a_i(k)}{k}\\
    \st&\quad\sum_{i=1}^N\sum_{k=1}^N a_i(k)=\sum_i\alpha_i,\quad \sum_{i=1}^N\sum_{k=1}^N k a_i(k)= \sum_i\gamma_i,\\
    &\quad \sum_i a_i(k)\ge 0,\quad k=1,\ldots,N.
\end{split}
\end{equation}
Since every feasible point of (\ref{LP_KAT}) is also a feasible point of (\ref{LP_DS_a_i_k}). The LP problem (\ref{LP_DS_a_i_k}) is a relaxed problem of (\ref{LP_KAT}). Therefore, $\lKAT\ge \lDS$.

Next, it is easy to show that $\lKAT\ge \lDC$ since based on the constraints of (\ref{LP_KAT}), one can get the DC bound as a lower bound of the objective function of (\ref{LP_KAT}) using the Cauchy-Schwarz inequality. Therefore, $\lKAT$ is lower bounded by $\lDC$.

Finally, we prove that $\lKAT\le\frac{9}{8}\lDC$.
Note that the DC bound (\ref{DC_bound}) is given by $\lDC=\sum_i\frac{\alpha_i^2}{\gamma_i}$ and that the solution of (\ref{LP_KAT_subproblem}) can be written as $\frac{a_i(k_1)}{k_1}+\frac{a_i(k_2)}{k_2}$ where $k_2=k_1+1$. It then suffices to prove for any $i=1,\dots,N$ and integer $k=1,\dots,N-1$
\begin{equation}\label{}
  \frac{a_i(k)}{k}+\frac{a_i(k+1)}{k+1}\le\frac{9}{8}\frac{\alpha_i^2}{\gamma_i},
\end{equation}
where $a_i(k)+a_i(k+1)=\alpha_i$ and $k a_i(k)+(k+1)a_i(k+1)=\gamma_i$.

Denoting $x:=a_i(k)$ and $y:=a_i(k+1)$, one can get
\begin{equation}\label{}
\begin{split}
  &\left(\frac{a_i(k)}{k}+\frac{a_i(k+1)}{k+1}\right)/\frac{\alpha_i^2}{\gamma_i}=\frac{\left(\frac{x}{k}+\frac{y}{k+1}\right)\left[kx+(k+1)y\right]}{\alpha_i^2}\\
  &=\frac{x^2+\frac{k}{k+1}xy+\frac{k+1}{k}xy+y^2}{(x+y)^2}=1+\frac{1}{k(k+1)}\frac{xy}{(x+y)^2}\\
  &\le 1+\frac{1}{4}\frac{1}{k(k+1)}\le 1+\frac{1}{8}=\frac{9}{8}.
\end{split}
\end{equation}
The first equality holds when $x=y=\frac{\alpha_i}{2}$ and the second equality holds when $k=1$. Therefore, the inequality $\lKAT\le\frac{9}{8}\lDC$ can be active.
\end{proof}

\begin{remark}Finally, we can derive an upper bound for $\ProbUnion$ using $\{P(A_i)\}$ and $\{\sum_j P(A_i\cap A_j)\}$ by maximizing the LP problem for the KAT bound.

The following LP problem
\begin{equation}\label{LP_KAT_UPPER}
\begin{split}
    \max_{\{a_i(k),i=1,\ldots,N,k=1,\ldots,N\}}&\quad\sum_{i=1}^N\sum_{k=1}^N\frac{a_i(k)}{k}\\
    \st&\quad\sum_{k=1}^N a_i(k)=\alpha_i,\quad \sum_{k=1}^N k a_i(k)= \gamma_i,\quad i=1,\ldots,N,\\
    &\quad a_i(k)\ge 0,\quad i=1,\ldots,N,\quad k=1,\ldots,N,
\end{split}
\end{equation}
gives the upper bound
\begin{equation}\label{upper1}
\begin{split}
    P\left(\bigcup_i A_i\right)&\le \sum_i \alpha_i-\frac{1}{N}\sum_i\beta_i\\
    &=\sum_i P(A_i)-\frac{1}{N}\sum_{j\neq i} P(A_i\cap A_j)=:\uNEWI.
\end{split}
\end{equation}
\end{remark}
\subsection{Lower Bounds Using ${\{P(A_i)\}}$ and ${\{P(A_i\cap A_j)\}}$}
Lower and upper bounds on $P\left(\bigcup_{i=1}^N A_i\right)$ in terms of the individual event probabilities $P(A_i)$'s and the pairwise event probabilities $P(A_i\cap A_j)$'s can be seen as special cases of the Boolean probability bounding problem \cite{Boros2014,Vizvari2004}, which can be solved numerically via a linear programming (LP) problem involving $2^N$ variables. Unfortunately, the number of variables for Boolean probability bounding problems increases exponentially with the number of events, $N$, which makes finding the solution impractical. Therefore, some suboptimal numerical bounds are proposed \cite{Boros2014,Vizvari2004,Prekopa2005,GalambosBook} in order to reduce the complexity of the LP problem, for example, by using the dual basic feasible solutions.

On the other hand, analytical/algorithmic bounds are particularly important. One can apply an existing bound using $\{P(A_i), i\in\mathcal{I}\}$ and $\{\sum_{j\in\mathcal{I}} P(A_i\cap A_j)\}$ as a base bound, and then optimize the bound by choosing the optimal subset $\mathcal{I}$ of $\{1,\dots,N\}$ algorithmically. Note that the bound by optimization via a subset exploits the full information of $\{P(A_i)\}$ and $\{P(A_i\cap A_j)\}$. Examples of bounds in this class includes the stepwise algorithmic implementation of the Kounias lower bound \cite{Kuai2000a}, the greedy algorithmic implementation of the Hunter upper bound \cite{Kuai2000a}. Other analytical bounds, like the KAT bound, are also investigated in other works (e.g., see \cite{Chen2006,Hoppe2006,Hoppe2009,Kuai2000a,Behnamfar2005,Behnamfar2007,Mao2013}).

The other class of bounds is established by $\{P(A_i)\}$ and $\{\sum_j c_j P(A_i\cap A_j)\}$, where $\{c_j\}$ can be arbitrarily chosen from a continuous set for $\bc=(c_1,\dots,c_N)^T$ or computed using $\{P(A_i)\}$ and $\{P(A_i\cap A_j)\}$. Then the resulting bound also exploits the full information of  $\{P(A_i)\}$ and $\{P(A_i\cap A_j)\}$. Typical example of the bounds in this class is the Gallot-Kounias (GK) bound \cite{Gallot1966,Kounias1968} (see also \cite{Feng2010,Mao2013}).

\subsubsection{Kounias Lower Bound and Algorithmic Implementation}
The Kounias lower bound, which is a Bonferroni-type bound, can be written as
\begin{equation}\label{kounias_lower_bound}
  \ProbUnion\ge \max_{\mathcal{I}}\left\{\sum_{i\in\mathcal{I}}P(A_i)-\sum_{i,j\in\mathcal{I}, i<j}P(A_i\cap A_j)\right\},
\end{equation}
where $\mathcal{I}$ is a subset of the set of indices $\{1,\dots,N\}$. However, the computational complexity of the Kounias lower bound is exponential since there are exponential number of subsets of $\{1,\dots,N\}$.

In order to reduce the computational complexity, an algorithmic algorithm is proposed in \cite{Kuai2000a} using a stepwise algorithm to find a sub-optimal index set that maximizes the RHS of (\ref{kounias_lower_bound}). We will refer to this algorithmic implementation of Kounias lower bound as the stepwise lower bound.

\subsubsection{Gallot-Kounias (GK) Bound}\label{subsection_GK_bound}
Let $\balpha=\left(P(A_1),\cdots,P(A_N)\right)^T\in\mathbb{R}^{N\times 1}$ and
\begin{equation}\label{}
    \bSigma=\left(
                                                                       \begin{array}{cccc}
                                                                         P(A_1\cap A_1) & P(A_1\cap A_2) & \dots & P(A_1\cap A_N) \\
                                                                         P(A_2\cap A_1) & P(A_2\cap A_2) & \dots & P(A_2\cap A_N) \\
                                                                         \vdots & \vdots & \ldots & \vdots \\
                                                                         P(A_N\cap A_1) & P(A_N\cap A_2) & \dots & P(A_N\cap A_N) \\
                                                                       \end{array}
                                                                     \right)\in\mathbb{R}^{N\times N},
\end{equation}
the GK bound \cite{Gallot1966,Kounias1968} is given as
\begin{equation}\label{}
    P\left(\bigcup_{i-1}^N A_i\right)\ge \bc^T\bSigma\bc,
\end{equation}
where $\bSigma\bc=\balpha$. Kounias has shown in \cite[Lemma 1.1]{Kounias1968} that the vector $\balpha$ is in the range of $\bSigma$, i.e., $\balpha$ is orthogonal to the null space of $\bSigma$. As a result, if $\bSigma$ is singular, one can choose subsets of $\{A_1,\dots,A_N\}$ to compute the corresponding GK bound, which results in the same bound if the rank of the corresponding $\bSigma$ is the same.

Therefore, without loss of generality (WLOG), we assume herein $\bSigma$ is non-singular, then the solution of $\bSigma\bc=\balpha$ is unique
\begin{equation}\label{}
    \tilde{\bc}=\bSigma^{-1}\balpha.
\end{equation}
and the GK bound can be written as
\begin{equation}\label{def_lGK}
     P\left(\bigcup_{i-1}^N A_i\right)\ge \tilde{\bc}^T\bSigma\tilde{\bc}=\balpha^T\left(\bSigma^T\right)^{-1}\balpha=\balpha^T\bSigma^{-1}\balpha=:\lGK,
\end{equation}
where $\left(\bSigma^T\right)^{-1}=\bSigma^{-1}$ as $\bSigma$ is symmetric.
Furthermore, the GK bound was recently revisited by \cite{Feng2010}. The authors in \cite{Feng2010} have shown that the GK bound can be reformulated as
\begin{equation}\label{}
  \lGK=\max_{\bc\in\mathbb{R}^N}\frac{\left[\sum_{i} c_i P(A_i)\right]^2}{\sum_{i}\sum_{k}c_ic_kP(A_i\cap A_k)}.
\end{equation}
\subsection{Upper Bounds}
There are only a few analytical/algorithmic upper bounds in the literature. In the following, we introduce the Hunter bound and its algorithmic implementation by a greedy algorithm.
\subsubsection{Hunter Upper Bound and Algorithmic Implementation}
The Hunter upper bound, which is a Bonferroni-type bound, can be written as
\begin{equation}\label{}
  \ProbUnion\le\sum_{i=1}^N P(A_i)-\max_{T_0\in\mathcal{T}}\sum_{(i,j)\in T_0} P(A_i\cap A_j),
\end{equation}
where $\mathcal{T}$ is the set of all trees spanning the $N$ indices, i.e., the trees that include all indices as nodes.

However, the computational complexity of finding the optimal spanning tree is exponential via an exhaustive search. In order to reduce the complexity, one algorithmic algorithm is proposed in \cite{Kuai2000a} using Kruskal's greedy algorithm for finding a sub-optimal spanning tree for a weighted graph. We will refer to this algorithmic implementation of the Hunter upper bound as the greedy upper bound, $\uGREEDY$.





\allowdisplaybreaks

\section{Observations on the GK Bound}
Finally, we conclude this survey with two observations on the GK bound.
\subsection{Applying the GK bound to subsets of events}
We note that many existing lower bounds, which do not fully explore available information, can be further improved algorithmically via optimization over subsets, as in \cite{Behnamfar2007,Hoppe2006}. However, in this section, we prove that the GK bound cannot be improved by applying it to subsets of $\{A_1,\dots,A_N\}$.

\begin{lemma} \label{prop1_commentsGK}
	For any given $M\ge N$, the GK bound is the solution of the following problem:
	\begin{equation}\label{temp_problem_commentsGK}
	\begin{split}
	&\min_{\{\bx\in\mathbb{R}^{M\times 1}, \bA\in\mathbb{R}^{N\times M}\}} \|\bx\|^2\\
	&\st\quad \bA\bx=\balpha,\quad \bA\bA^T=\bSigma\in\mathbb{R}^{N\times N}.
	\end{split}
	\end{equation}
\end{lemma}
\begin{proof}
	We can always write $\bx$ as
	\begin{equation}\label{}
	\bx=\bA^T\bk_1+\bA_{\perp}^T\bk_2,
	\end{equation}
	where $\bk_1\in\mathbb{R}^{N\times 1}$, $\bk_2\in\mathbb{R}^{(M-N)\times 1}$, and $\bA_{\perp}\in\mathbb{R}^{(M-N)\times M}$ satisfies $\bA\bA_{\perp}^T=\bzero_{N\times(M-N)}$.
	Particularly, let $\bL\bL^T$ be the Cholesky decomposition of $\bSigma$, then $\bA=\left(\bL,\bzero_{N\times(M-N)}\right)\bQ$ where $\bQ$ is any orthogonal matrix is the solution of (\ref{temp_problem_commentsGK}).
	
	Then
	\begin{equation}\label{}
	\|\bx\|^2=\bk_1^T\bA\bA^T\bk_1+\bk_2^T\bA_{\perp}\bA_{\perp}^T\bk_2=\bk_1^T\bSigma\bk_1+\bk_2^T\bA_{\perp}\bA_{\perp}^T\bk_2.
	\end{equation}
	The first constraint $\bA\bx=\balpha$ implies $\bA\bA^T\bk_1=\balpha$, i.e., $\bk_1=\bSigma^{-1}\balpha$. Therefore, the minimum of $\|\bx\|^2$ is achieved at $\bk_2=\bzero$, so that
	\begin{equation}\label{}
	\min_{\bk_1,\bk_2} \|\bx\|^2=\bk_1^T\bSigma\bk_1=\balpha^T\bSigma^{-1}\balpha.
	\end{equation}
\end{proof}

\begin{theorem}\label{prop2_commentsGK}
	The GK bound cannot be improved via optimization over subsets of $\{A_1,A_2,\dots,A_N\}$.
\end{theorem}
\begin{proof}
	Denoting $\bA=\left(\begin{array}{c}
	\ba_1^T\\
	\ba_2^T\\
	\vdots\\
	\ba_N^T\\
	\end{array}
	\right)$, the first constraint of (\ref{temp_problem_commentsGK}) is equivalent to $N$ constraints $\ba_i^T\bx=\alpha_i, i=1,\dots,N$, and the second constraint of (\ref{temp_problem_commentsGK}) is equivalent to $N^2$ constraints $\ba_i^T\ba_j=\bSigma_{ij}, i=1,\dots,N, j=1,\dots, N$.
	
	By selecting a subset of $\{1,2,\dots,N\}$, the resulting GK bound is the solution of a relaxed problem of (\ref{temp_problem_commentsGK}) with a subset of constraints on $\bx$ and $\ba_i, i=1,\dots, N$. Since the objective value of the relaxed problem must be no more than the original problem (\ref{temp_problem_commentsGK}), the GK bound using a subset cannot be higher than the GK bound using full information.
\end{proof}
\subsection{Iterative implementation of the GK bound}
\begin{theorem}\label{prop3_commentsGK}
	The GK bound can be computed iteratively. More specifically, for $n=1,\dots,N$, denote $\lGK(n)$ as the GK bound using the information of $A_1,\dots, A_n: \balpha_n=\left(P(A_1),\dots,P(A_n)\right)^T$, $\bbeta_n=\left(P(A_1\cap A_n),\dots,P(A_{n-1}\cap A_n)\right)^T$ and $\bSigma_{n}$, the $n\times n$ upper left sub-matrix of $\bSigma$, then if $\alpha_n-\bbeta_n^H\bSigma_{n-1}^{-1}\bbeta_n>0$, we have
	\begin{equation}\label{temp_eq_commentsGK}
	\lGK(n)=\lGK(n-1)+\frac{1}{\alpha_n-\bbeta_n^H\bSigma_{n-1}^{-1}\bbeta_n}\balpha_n^T\left(
	\begin{array}{cc}
	\bb_n\bb_n^T & \bb_n \\
	\bb_n^T & 1 \\
	\end{array}
	\right)
	\balpha_n,
	\end{equation}
	where $\bb_n=-\bSigma_{n-1}^{-1}\bbeta_n$.
	
	The matrix $\bSigma_{n}$ is not invertible if and only if $\alpha_n-\bbeta_n^H\bSigma_{n-1}^{-1}\bbeta_n=0$. The last case $\alpha_n-\bbeta_n^H\bSigma_{n-1}^{-1}\bbeta_n<0$ never happens, i.e., the following inequality holds
	\begin{equation}\label{invertible_condition_commentsGK}
	\begin{split}
	P(A_n)&\ge\left(P(A_1\cap A_n),\dots,P(A_{n-1}\cap A_n)\right)\cdot \\
	&\left(
	\begin{array}{ccc}
	P(A_1\cap A_1)  & \dots & P(A_1\cap A_{n-1}) \\
	P(A_2\cap A_1)  & \dots & P(A_2\cap A_{n-1}) \\
	\vdots &  \ldots & \vdots \\
	P(A_{n-1}\cap A_1) & \dots & P(A_N\cap A_{n-1}) \\
	\end{array}
	\right)^{-1}       \left(\begin{array}{c}
	P(A_1\cap A_n)\\
	P(A_2\cap A_n)\\
	\vdots\\
	P(A_{n-1}\cap A_n)\\
	\end{array}
	\right),
	\end{split}
	\end{equation}
	for $n=1,\dots,N$.
\end{theorem}
\begin{proof}
	Note that
	\begin{equation}\label{}
	\bSigma_n=\left(
	\begin{array}{cc}
	\bSigma_{n-1} & \bbeta_n \\
	\bbeta_n^T & \alpha_n \\
	\end{array}
	\right).
	\end{equation}
	and by the matrix inverse lemma for a Hermitian matrix \cite{HornBook1986}, we have
	\begin{equation}\label{denominator_commentsGK}
	\bSigma_n^{-1}=\left(
	\begin{array}{cc}
	\bSigma_{n-1}^{-1} & \bzero_n \\
	\bzero_n^T & 0 \\
	\end{array}
	\right)+\frac{1}{\alpha_n-\bbeta_n^H\bSigma_{n-1}^{-1}\bbeta_n}\left(
	\begin{array}{cc}
	\bb_n\bb_n^T & \bb_n \\
	\bb_n^T & 1 \\
	\end{array}
	\right)
	\end{equation}
	Substituting to $\lGK(n)=\balpha_n\bSigma_n^{-1}\balpha_n$, $\lGK(n)$ can be computed using $\lGK(n-1)$.
	
	Since we have proved $\lGK(n)\ge \lGK(n-1)$ in Theorem \ref{prop2_commentsGK}, the inequality is directly from $\alpha_n-\bbeta_n^H\bSigma_{n-1}^{-1}\bbeta_n\ge 0$ in (\ref{temp_eq_commentsGK}) of Theorem \ref{prop3_commentsGK}.
\end{proof}

\noindent
{\bf Recent works:} We close this survey by referring the reader to \cite{Yang2014ISIT,Yang2015b,Yang2015,Yang2016,Yang2016a} for the most recent works on this topic.

\printbibliography




\end{document}